\documentclass[a4paper,11pt,draft]{amsart}

\usepackage[left=2.3cm,right=2.3cm,top=3.5cm,bottom=3cm]{geometry}

\usepackage{amsmath,amssymb,amsthm,amscd,amsfonts}
\usepackage{mathrsfs}
\usepackage{latexsym}
\usepackage[all]{xy}
\usepackage{verbatim}
\usepackage[usenames, dvipsnames]{color}
\usepackage{MnSymbol}
\usepackage{multirow}
\allowdisplaybreaks

 \newtheorem{thm}{Theorem}[section]
 
 \newtheorem{lem}[thm]{Lemma}
 
 \theoremstyle{definition}
 \newtheorem{defn}[thm]{Definition}
 \theoremstyle{remark}
 \newtheorem{rem}[thm]{Remark}
 \newtheorem*{ex}{Example}
 \numberwithin{equation}{section}

\newcommand{\N}{\mathbb{N}}
\newcommand{\Z}{\mathbb{Z}}
\newcommand{\C}{\mathbb{C}}
\newcommand{\F}{\mathbb{F}}
\renewcommand{\P}{\mathbb{P}} 
\newcommand{\G}{\Gamma} 
\newcommand{\Tr}{\mathcal{T}} 
\renewcommand{\c}{\mathbf{c}} 
\newcommand{\I}{\mathcal{I}} 
\newcommand{\g}{\gamma}

\renewcommand{\b}{\beta}

\newcommand{\e}{\overline{e}}

\renewcommand{\O}{\Omega}
\renewcommand{\o}{\omega}
\newcommand{\Hom}{\mathrm{Hom}}

\newcommand{\act}[1]{\mathop{\big\|}_{\scriptstyle{#1}}}

\renewcommand{\matrix}[4]{\left(\begin{array}{cc} {#1} & {#2} \\ {#3} & {#4} \end{array}\right)}

\begin{document}

\title[Atkin $U_t$-operator for $\Gamma_1(t)$-invariant Drinfeld cusp
forms]{On the Atkin $U_t$-operator for $\Gamma_1(t)$-invariant Drinfeld cusp
forms}

\author{Andrea Bandini}

\address{{\sc Andrea Bandini}: Universit\`a degli Studi di Parma \\
   Dipartimento di Scienze Matematiche, Fisiche e Informatiche \\
   Parco Area delle Scienze, 53/A \\
   43124 Parma - Italy
   }

\email{andrea.bandini@unipr.it}

\author{Maria Valentino}\thanks{M. Valentino is supported by an outgoing Marie-Curie fellowship of INdAM}
   \address{{\sc Maria Valentino}: Universit\`a degli Studi di Parma \\
   Dipartimento di Scienze Matematiche, Fisiche e Informatiche \\
   Parco Area delle Scienze, 53/A \\
   43124 Parma - Italy}

   \email{maria.valentino@unipr.it}

   \subjclass[2010]{Primary 11F52, 11F25; Secondary 11B65, 20E08, 11C20.}

   \keywords{Drinfeld cusp forms, harmonic coycles, Atkin-Lehner operator, diagonalizability}

\begin{abstract}
We study the diagonalizability of the Atkin $U_t$-operator acting on Drinfeld cusp forms for $\G_1(t)$ and $\G(t)$ using
Teitelbaum's interpretation as harmonic cocycles. For small weights $k\leqslant 2q$, we prove $U_t$ is diagonalizable in
odd characteristic and we point out that non diagonalizability in even characteristic depends on 
antidiagonal blocks.\end{abstract}

\maketitle

\section{Introduction}
Let $F:=\mathbb{F}_{q}(t)$, with $q=p^r$ and $p\in \Z$ a prime, and 
let $A:=\F_q[t]$ be the ring of functions regular outside $\infty:=\frac{1}{t}$.
Let $F_\infty$ be the completion of $F$ at $\infty$ and $\C_\infty$ be the completion of an algebraic closure of $F_\infty$.
Drinfeld cusp forms (defined over the Drinfeld upper half plane $\O:= \P^1(\C_\infty) - \P^1(F_\infty)\,$) 
admit a natural action of {\em Hecke operators} ${\bf T}_\mathfrak{p}$ (as $\mathfrak{p}$ varies
among the primes of $A$). In this paper we deal with the action of $U_t:={\bf T}_{(t)}$ on
$S^1_{k,m}(\G)$ (the space of $\G$-invariant cusp forms of weight $k$ and type $m$) for the congruence groups
$\G=\G_1(t),\G(t)$ of {\em level} $t$. 

In this context the question about the diagonalizability of the Hecke operators is still open, due mainly to the
lack of an adequate analogous of Petersson inner product.
For the operators ${\bf T}_{\mathfrak{p}}$ with $\mathfrak{p}\neq (t)$ generated by a polynomial of degree 1, 
Li and Meemark in \cite{LM} checked diagonalizability for $k\leqslant q+3$, i.e., until they found the first
example of a non diagonalizable matrix (in even characteristic) together with the presence of an inseparable eigenvalue
(see \cite{LM}, in particular, the Remark in page 1951).
B\"ockle and Pink  computed the structure of double cusp forms for $\G_1(t)$ for some fixed $k$ and $q$ 
(\cite[Section 15]{B}) and for those of weight 4 they determined all eigenvalues (\cite[Proposition 15.6]{B}).

Using the Bruhat-Tits tree $\Tr$ as a combinatorial counterpart for $\O$, Teitelbaum in \cite{T}
provided a reinterpretation of cusp forms as $\G$-invariant harmonic cocycles.
In \cite{BVHP1} (inspired by the computations of \cite{LM}), we used this interpretation
to get the matrix corresponding to the action of $U_t$ on $S^1_{k,m}(\G_1(t))$, but then focused only on
the blocks of that matrix associated to cusp forms for $\G_0(t)$. Here we consider the whole matrix: a careful study of
the coefficients will allow us to discuss the diagonalizability of $U_t$ in detail for weights $k\leqslant 2q$.

\noindent After recalling the basic definitions and properties of the objects we shall work with,
in Section \ref{SecGamma1} we consider the action of $U_t$ on of $S^1_k(\G_1(t))$ with respect to the basis
$\mathcal{B}^1_k(\G_1(t)):=\{\c_j(\e)\,,\,0\leqslant j\leqslant k-2\}$. The associated matrix has (at most) 
$q-1$ blocks (one for each residue class modulo $q-1$) and $U_t$ is diagonalizable if and only if each block is. 
We denote by $M_j$ ($0\leqslant j\leqslant q-2$) the block that describes the action of $U_t$ on the subspace with basis
$C_j:=\{\c_\ell\in \mathcal{B}^1_k(\G_1(t))\,:\,\ell\equiv j \pmod{q-1}\}$.
Then we study diagonalizability of the $M_j$ for $k\leqslant q+3$. 
In particular we show that (see Theorems \ref{Thm1} and \ref{Thm1G1})\\
\indent{\bf 1.} if $k\leqslant q+2$, then $U_t$ is diagonalizable;\\
\indent{\bf 2.} if $k=q+3$, then $U_t$ is diagonalizable if and only if $q$ is odd.

\noindent Non diagonalizability here depends on the antidiagonal form of one of the $M_j$ (which also leads to the presence of
an inseparable eigenvalue). We have already encountered this phenomenon in \cite{BVHP1} while studying the action
of $U_t$ on cusp forms for $\G_0(t)$. Among the $M_j$ there are two blocks (if $q$ is odd, just one if $q$ is even)
which depend on $\G_0(t)$-invariant cusp forms (see \cite[Section 4.3]{BVHP1}) and the action on newforms (i.e., those
of proper level $t$) has the tendency to being antidiagonal (see \cite[Theorem 5.2]{BVHP1} for a partial result
in this direction). Hence we believe those blocks would be diagonalizable only for $q$ odd but it is interesting
anyway to check diagonalizability of the remaining ones.

\noindent In Section \ref{SecDim2} we study  all blocks of dimension 2 and obtain (see Theorem \ref{ThmDiag2OddChar}
and Section \ref{SecDiagp=2})

\begin{thm}
Assume $|C_j|=2$, then $M_j$ is diagonalizable unless $q$ is even, $k$ is odd and $M_j$ is antidiagonal.
\end{thm}

\noindent Finally, in Section \ref{SecGamma} we study diagonalizability of $U_t$ on $S^1_k(\G(t))$. We show that
$U_t$ has a large kernel which actually reduces its diagonalizability to the previous case 
(see Theorem \ref{ThmDiagGammaGamma1}).

\section{Setting and notations}\label{SecNotations}
As above let $F$ be the rational function field $F =\mathbb{F}_{q}(t)$, fix 
$\frac{1}{t}$ as the prime at $\infty$ and let $A:=\F_q[t]$ be the ring of functions regular outside $\infty$.
Let $F_\infty = \F_q((1/t))$ be the completion of $F$ at $\infty$ and let $\C_\infty$ denote 
the completion of an algebraic closure of $F_\infty$.
The {\em Drinfeld upper half-plane} is the set $\O:=\P^1(\C_\infty) - \P^1(F_\infty)$
together with a structure of rigid analytic space (see \cite{FvdP}).

\subsection{The Bruhat-Tits tree}\label{SecTree}
The Drinfeld's upper half plane has a combinatorial counterpart, the {\em Bruhat-Tits tree} $\Tr$ of $GL_2(F_\infty)$, which
 is a $(q+1)$-regular tree on which $GL_2(F_\infty)$ acts transitively (see, e.g.,
\cite{G3} or \cite{S1}). Let $Z(F_\infty)$ be the center
of $GL_2(F_\infty)$ and let $\I(F_\infty)$ be the {\em Iwahori subgroup}, i.e.,
\[ \I(F_\infty)=\left\{\matrix{a}{b}{c}{d} \in GL_2(A_\infty)\, : \, c\equiv 0 \pmod \infty \right\}\,. \]
Then the sets $X(\Tr)$ of vertices and $Y(\Tr)$ of oriented edges of $\Tr$ are given by
\[ X(\Tr)= GL_2(F_\infty)/Z(F_\infty)GL_2(A_\infty)\quad \mathrm{and}\quad Y(\Tr)= GL_2(F_\infty)/ Z(F_\infty)\I(F_\infty)\,. \]

Two infinite paths in $\Tr$ are considered equivalent if they differ at finitely many edges. An {\em end} is an equivalence
class of infinite paths: the ends identify points in $\P^1(F_\infty)$ via a $GL_2(F_\infty)$-equivariant bijection.
For any arithmetic subgroup $\G$ of $GL_2(A)$ the {\em cusps} of $\G$ are the elements of $\G\backslash \P^1(F)$, which
are in bijection with the ends of $\G\backslash \Tr$ (see \cite[Proposition 3.19]{B}).\\
Following Serre \cite[pag 132]{S1}, we call a vertex or an edge {\em $\G$-stable} if its stabilizer in $\G$ is trivial
and {\em $\G$-unstable} otherwise.

\subsection{Harmonic cocycles}\label{SecHarCoc}
For $k\geqslant 0$ and $m\in\Z$, let $V(k,m)$ be the $\C_\infty$ vector space generated by 
$\{X^jY^{k-2-j}: 0\leqslant j\leqslant k-2 \}$. The action of 
$\gamma= \left( \begin{smallmatrix} a & b  \\ c & d \end{smallmatrix} \right)\in GL_2(F_\infty)$ 
on $V(k,m)$ is given by
\[ \g(X^jY^{k-2-j}) = \det(\g)^{m-1}(dX-bY)^j(-cX+aY)^{k-2-j}\quad {\rm for}\ 0\leqslant j\leqslant k-2\,.\]
For every $\o\in \Hom(V(k,m),\C_\infty)$ we have an induced action of $GL_2(F_\infty)$
\[ (\g\o)(X^jY^{k-2-j})=\det(\g)^{1-m}\o((aX+bY)^j(cX+dY)^{k-2-j})\quad {\rm for}\ 0\leqslant j\leqslant k-2\,. \]

\begin{defn}
A {\em harmonic cocycle of weight $k$ and type $m$} for $\G$ is a function $\c$ from the set of directed edges
of $\Tr$ to $\Hom(V(k,m),\C_\infty)$ satisfying:
\begin{itemize}
\item[{\bf 1.}] ({\em harmonicity}) for all vertices $v$ of $\Tr$,
$\displaystyle{\sum_{t(e)= v}\c(e)=0}$,
where $e$ runs over all edges in $\Tr$ with terminus $v$;
\item[{\bf 2.}] ({\em antisymmetry}) for all edges $e$ of $\Tr$, $\c(\overline{e})=-\c(e)$, where $\overline{e}$ is the edge $e$
with reversed orientation;
\item[{\bf 3.}] ({\em $\G$-equivariancy}) for all edges $e$ and elements $\g\in\G$, $\c(\g e)=\g(\c(e))$.
\end{itemize}
\end{defn}

\noindent The space of harmonic cocycles of weight $k$ and type $m$ for $\G$ will be denoted by $C^{har}_{k,m}(\G)$.
By \cite[Lemma 20]{T}, cocycles in $C^{har}_{k,m}(\G)$ are determined by their
values on the stable (non-oriented) edges of a fundamental domain.

\subsection{Drinfeld modular forms}\label{SecDrinfModForms}
An element $\gamma= \left( \begin{smallmatrix} a & b  \\ c & d \end{smallmatrix} \right)\in GL_2(A)$
acts on $\Omega$ via M\"obius trasformation and for $k,m \in \Z$
and $f:\O\to \C_\infty$, we define
\[ (f \act{k,m}\g)(z) := f(\g z)(\det \g)^m(cz+d)^{-k}. \]

\begin{defn}
A rigid analytic function $f:\O\to \C_\infty$ is called a {\em Drinfeld modular function of weight $k$ and type $m$} 
for an arithmetic group $\G$, if
\begin{equation}\label{Mod} 
(f \act{k,m} \g )(z) =f(z)\ \ \forall \g\in\G\,.  
\end{equation}
A Drinfeld modular function $f$ for $\G$ is called {\em Drinfeld modular form} if
$f$ is holomorphic at all cusps of $\G$.
A Drinfeld modular form $f$ for $\G$ is called {\em cusp form} (resp. {\em double cusp form}) if it vanishes at all
cusps to the order at least 1 (resp. to the order at least 2).
\end{defn}

\noindent The space of Drinfeld cuspidal modular forms (resp. doubly cuspidal)
of weight $k$ and type $m$ for $\G$ will be denoted by $S^1_{k,m}(\G)$ (resp. $S^2_{k,m}(\G)\,$). 
When all elements of $\Gamma$ have determinant 1 (as will happen in all our computations), the type does not play a role. 
In this case all $S^1_{k,m}(\G)$ ($m\in \Z$) are isomorphic and we shall denote them simply by $S^1_k(\Gamma)$ 
(same for $S^2_k(\G)$).

\subsubsection{Cusp forms and harmonic cocycles}\label{SecIsomModFrmHarCoc}
In \cite{T}, Teitelbaum constructed the so-called {\em residue map} which provides an isomorphism 
$S_{k,m}(\G)\simeq C^{har}_{k,m}(\G)$ (\cite[Theorem 16]{T}). For more details the reader is referred 
to the original paper of Teitelbaum or to \cite[Section 5.2]{B} which is full of details written in 
a more modern language.

\subsection{The Hecke operator $U_t$}\label{SecHecke}
Hecke operators on Drinfeld modular forms are formally defined using a double coset decomposition 
(see, e.g., \cite{A}). Here we just provide the definition of $U_t$ on harmonic cocycles
which is more suitable for computation. 

\noindent Our Atkin $U_t$-operator is
\[ U_t(f)(z):=\sum_{\beta\in \F_q} f\left(\frac{z+\beta}{t}\right) \]
(as in \cite{LM} and \cite[Section 4]{BVHP1}, we normalize the classical $U_t$ multiplying it by $t^{k-m}$ to eliminate 
any reference to the type $m$ in the final formulas). Via the residue map one translates this action on 
harmonic cocycles (for details see formula (17) in \cite[Section 5.2]{B}) 
\begin{align*}
U_t(\c(e))= t^{k-m}\sum_{\beta\in \F_q}  {\matrix{1}{\beta}{0}{t}}^{-1}\c\left( \matrix{1}{\beta}{0}{t}e\right) \,.
\end{align*}
We focus on the two congruence groups 
\[ \G(t):=\left\{ \gamma \in GL_2(A): \g\equiv Id \pmod{t} \right\}\quad \mathrm{and}\quad
\G_1(t):=\left\{ \g \in GL_2(A): \g\equiv \left( \begin{smallmatrix} 1&*\\0&1 \end{smallmatrix} \right) \pmod{t} \right\} \]
(for $\G_0(t)$ see \cite{BVHP1}).

\section{Action of $U_t$ on cusp forms for $\G_1(t)$} \label{SecGamma1}
The tree $\G_1(t)\backslash \Tr$ has two cusps (\cite[Proposition 5.6]{GN}) corresponding to $[1:0]$ and $[0:1]$.
The path connecting the two cusps is a fundamental domain for $\G_1(t)$. Here is a picture of it
\[ \xymatrix { \dots v_{-2,0}=\matrix {1}{0}{0}{t^2} \ar@/^0.5pc/@{<-}[r]^{\overline{e}_{-2,0}=\matrix {0}{1}{t}{0}}
\ar@/_0.5pc/[r]_{e_{-2,0}=\matrix {1}{0}{0}{t^2}} &
v_{-1,0}=\matrix {1}{0}{0}{t} \ar@/^0.5pc/@{<-}[r]^{\overline{e}_{-1,0}=\matrix {0}{1}{1}{0}}
\ar@/_0.5pc/[r]_{e_{-1,0}=\matrix {1}{0}{0}{t}} &
v_{0,0}=\matrix {1}{0}{0}{1} \ar@/^0.5pc/@{<-}[r]^{\overline{e}_{0,0}=\matrix {0}{t}{1}{0}}
\ar@/_0.5pc/[r]_{e_{0,0}=\matrix {1}{0}{0}{1}} &
v_{1,0}=\matrix {t}{0}{0}{1}\dots  } \] 
This fundamental domain does not contain stable vertices but one stable edge, namely $\e:=\overline{e}_{-1,0}$.

\subsection{The action of $U_t$} \label{UBasis}
Since $\G_1(t)$ has no prime to $p$ torsion and determinant 1, by
\cite[Propositions 5.4 and 5.18]{B}, the dimensions of $S^1_k(\G_1(t))$ and $S^2_k(\G_1(t))$
depend only on the genus of $\G_1(t)\backslash \Tr$ (which is 0 by \cite[Corollary 5.7]{GN}) 
and on the number of cusps. In particular
\begin{itemize}
\item[{\bf 1.}] {$\dim_{\C_\infty} S^1_k(\G_1(t))=k-1$;}
\item[{\bf 2.}] {$\dim_{\C_\infty} S^2_k(\G_1(t))= \left\{ \begin{array}{cr}
0 & k=2 \\
k-3 & k>2
\end{array}\right.$\ .}
\end{itemize}
We shall check diagonalizability of the operator $U_t$ starting with its action on a fixed basis: the same used
in \cite{LM} and \cite{BVHP1}.

For any $j\in\{0,1,\dots,k-2\}$, let $\c_j(\e)$ be defined by
\[ \c_j(\e)(X^iY^{k-2-i})=\left\{ \begin{array}{ll} 1 & {\rm if}\ i=j \\
0 & {\rm otherwise} \end{array} \right. \ .\]
The sets $\mathcal{B}^1_k(\G_1(t)):=\{\c_j(\e)\,,\,0\leqslant j\leqslant k-2\}$ and
$\mathcal{B}^2_k(\G_1(t)):=\{\c_j(\e)\,,\,1\leqslant j\leqslant k-3\}$ are bases for
$S^1_k(\G_1(t))$ and $S^2_k(\G_1(t))$ respectively. We shall work mainly on $\mathcal{B}^1_k(\G_1(t))$, the results for
$\mathcal{B}^2_k(\G_1(t))$ will easily follow and we shall point them out in some remarks along the way.

\noindent The action of $U_t$ on the basis $\mathcal{B}^1_k(\G_1(t))$ has been computed in \cite{BVHP1} 
and the final formula is 
\begin{align}\label{Ttcj}
U_t(\c_j(\e)) & = -(-t)^{j+1} \binom{k-2-j}{j} \c_j(\e) -t^{j+1}\sum_{h\neq 0}\left[ \binom{k-2-j-h(q-1)}{-h(q-1)} \right.\\
\ &\left. + (-1)^{j+1} \binom{k-2-j-h(q-1)}{j} \right] \c_{j+h(q-1)}(\e)  \nonumber
\end{align}
(where it is understood that $\c_{j+h(q-1)}(\e) \equiv 0$ whenever $j+h(q-1)<0$ or $j+h(q-1)>k-2$).

From formula \eqref{Ttcj} one immediately notes that the $\c_j$ can be divided into classes modulo $q-1$ and
every such class is stable under the action of $U_t$.
We shall denote by $C_j$ the class of $\c_j(\e)$, i.e, $C_j=\{\c_\ell(\e):\ell\equiv j\pmod{q-1}\}$. The cardinality
of $C_j$ is the largest integer $n$ such that $j+(n-1)(q-1)\leqslant k-2$. Reordering the basis as
$\mathcal{B}^1_k(\G_1(t))=\{C_0,C_1,\dots,C_{q-2}\}$, the matrix associated to the action of $U_t$ has (at most) $q-1$
blocks (of dimensions $|C_j|$, $0\leqslant j\leqslant q-2$) on the diagonal and 0 everywhere else. Obviously
the matrix is diagonalizable if and only if each block is. In particular $|C_j|\leqslant 1$ (i.e., $j+(q-1)>k-2$)
always yields a diagonal (or empty) block.

\subsection{Diagonalizability of $U_t$: $k\leqslant q+3$} \label{SecUkLessq+3}
From now on, to shorten proofs and notations we shall drop the $\e$ and we shall mainly consider classes
of cardinality $\geqslant 2$.

\begin{thm}\label{Thm1}
If $k\leqslant q+2$, then $U_t$ is diagonalizable.
\end{thm}

\begin{proof}
\noindent\underline{\em Case $k\leqslant q$} (the {\em trivial} case).\\
All classes $C_j$ have cardinality $\leqslant 1$. Equation \eqref{Ttcj} reduces to
\[U_t(\c_j) = -(-t)^{j+1} \binom{k-2-j}{j} \c_j \quad {\rm for\ any\ } 0\leqslant j \leqslant k-2\,, \]
so each $\c_j$ (for $0\leqslant j\leqslant \min\{q-2,k-2\}$) is an eigenvectors of eigenvalue
$-(-t)^{j+1} \binom{k-2-j}{j}$.

\noindent\underline{\em Case $k=q+1$.}\\
The unique class of cardinality $>1$ is $C_0=\{\c_0,\c_{q-1}\}$. Formula \eqref{Ttcj} yields
\[ U(\c_0)=t\left( \c_0 + \c_{q-1}\right)\quad \mathrm{and}\quad U(\c_{q-1})=-t^q\left( 1 + (-1)^q\right) \c_0=0\,.\]
The associated matrix has eigenvalues $t$ and $0$ with eigenvectors $\c_0+\c_{q-1}$ and $\c_{q-1}$.

\noindent\underline{\em Case $k=q+2$.}\\
We need to distinguish between $q=2$ and $q\geqslant 3$.

\noindent If $q\neq 2$ we have two more classes of cardinality 2, namely $C_0=\{\c_0,\c_{q-1} \}$ and $C_1=\{ \c_1,\c_q\}$.
For the class $C_0$ we have the same matrix of the previous case.
For the class $C_1$ one has 
\[ U\c_1=t^2 \c_1\quad \mathrm{and} \quad U\c_q=-t^{q+1}\c_1\,.\]
The associated matrix has eigenvalues $t^2$ and $0$ with eigenvectors $\c_1$ and $t^{q-1}\c_1+\c_q$.

\noindent If $q=2$ we just have one class (as it always happens with $q=2$), namely $C_0=\{\c_0,\c_1,\c_2 \}$.
The action on $C_0$ is given by
$U\c_0=t\left( \c_0 + \c_1 + \c_2 \right)$, $U\c_1=t^2 \c_1$ and $U\c_2=t^3\c_1$.
The associated matrix has eigenvalues $t^2$, $t$ and 0 with eigenvectors $\c_1$, $\c_0+(1+t)\c_1+\c_2$ and $t\c_1+\c_2$.
\end{proof}

\begin{rem}
For double cusp forms just observe that when $k\leqslant q+2$ all classes have cardinality (at most) 1.
\end{rem}

The first non diagonalizable matrix appears for $q$ even and $k=q+3$.

\begin{thm}\label{Thm1G1}
Let $k=q+3$, then $U_t$ is diagonalizable if and only if $q$ is odd.
\end{thm}

\begin{proof}
To emphasize the classes modulo $q-1$, we write
\[ k-2=q+1 =\left\{ \begin{array}{ll} (q-1)+2 & {\rm if}\ q\geqslant 4 \\
2(q-1) & {\rm if}\ q=3 \\
3(q-1) & {\rm if}\ q=2 \end{array} \right. \ .\]

\noindent\underline{\em Case $q\geqslant 4$.}\\
The (nontrivial) classes are $C_0:=\{\c_0,\c_{q-1}\}$, $C_1:=\{\c_1,\c_q\}$ and $C_2:=\{\c_2,\c_{q+1}\}$.\\
For the class $C_0$ we obtain the same matrix of Theorem \ref{Thm1} case $k=q+1$.\\
For the class $C_1$ we obtain
\[ U_t\c_1 = -t^2\c_q\quad \mathrm{and} \quad U_t\c_q = -t^{q+1}\c_1\,,\]
with associated antidiagonal matrix
\[ M_1=\matrix{0}{-t^{q+1}}{-t^2}{0}\,.\]
This is diagonalizable if and only if $q$ is odd. Indeed the eigenvalues for odd $q$ are $\pm \sqrt{t^k}$ (resp. $\sqrt{t^k}$
with multiplicity 2 if $q$ is even) with eigenvectors $\pm t^{\frac{q-1}{2}} \c_1+\c_q$ (resp.
$t^{\frac{q-1}{2}} \c_1+\c_q$ if $q$ is even).\\
For the class $C_2$ we have
\[ U_t\c_2 =t^3\c_2\quad \mathrm{and} \quad U_t\c_{q+1} = -t^{q+2}\c_2\,.\]
The associated matrix has eigenvalues $t^3$ and $0$ with eigenvectors $\c_2$ and $t^{q-1}\c_2+\c_{q+1}$.

\noindent\underline{\em Case $q=3$.}\\
Since $k-2=2(q-1)$ we have only two classes: $C_0=\{\c_0,\c_2,\c_4\}$ and
$C_1=\{\c_1,\c_3\}$.\\
For the class $C_0$ we have
\[ U_t\c_0 = t\c_0 + t\c_2+ t\c_4,\quad U_t\c_2 = t^3\c_2\quad \mathrm{and} \quad U_t\c_4 = -t^5\c_2\,.\]
The associated matrix has eigenvalues $t^3$, $t$ and 0 with eigenvectors $\c_2$, $\c_0+(1+t^2)\c_2+\c_4$ 
and $t^2\c_2+\c_4$.\\
For $C_1$ we have 
\[ U_t\c_1 = -t^2\c_3\quad  \mathrm{and} \quad U_t\c_3 = -t^4\c_1\,.\]
The associated matrix has eigenvalues $\pm t^3$ with eigenvectors $\mp t\c_1+\c_3$.

\noindent\underline{\em Case $q=2$}.\\
Here $k=5$ and we have only one class, namely $C_0:=\{\c_0,\dots,\c_3\}$.
With calculations as above, it is easy to see that the associated matrix is
\[ M_0=\left( \begin{array}{cccc} t & 0 & 0 & 0 \\ t & 0 & t^3 & t^4 \\ t & t^2 & 0 & t^4 \\
t & 0 & 0 & 0  \end{array} \right) \]
with characteristic polynomial $X(X+t)(X^2+t^5)$.
There is the inseparable eigenvalue $\sqrt{t^5}$ with multiplicity 2 and the matrix is not
diagonalizable.
\end{proof}

\begin{rem}
When dealing with double cusp forms for $k=q+3$ we loose the elements $\c_0$ and $\c_{q+1}$ and we
have only the class $C_1$ (which coincides with $C_0$ for $q=2$) with two elements. In particular:
\begin{itemize}
\item[$q\geqslant 3$] the class $C_1$ is untouched with its antidiagonal matrix which is diagonalizable if and only
if $q$ is odd;
\item[$q=2$] we loose the external frame of the matrix and are left
with the antidiagonal matrix $\left(\begin{smallmatrix} 0&t^3\\t^2&0\end{smallmatrix}\right)$. \end{itemize}
\end{rem}

The result for $k=q+3$ recalls the one obtained in \cite{LM} for the operators ${\bf T}_{\mathfrak{p}}$ (for $\mathfrak{p}\neq (t)$).
As $k$ grows it becomes quite difficult to check diagonalizability of the matrices $M_j$. The particular case of the
classes $C_j$ containing $\Gamma_0(t)$-invariant cusp forms has been treated in \cite{BVHP1}. For the other classes we can
provide a general result only for small dimensions.

\section{Diagonalizability of $U_t$: blocks of dimension 2} \label{SecDim2}
We shall consider only blocks arising from classes $C_j$ of order 2. Since some cases have already been
checked we assume that $k > q+3$, hence, in particular, we do not consider here the case $q=2$ because the unique class
$C_0$ would have cardinality $k-1>4$. Moreover for $q$ even we have that the class arising from $\Gamma_0(t)$ is
$C_{\frac{k-1-q}{2}}$ ($k$ odd, see \cite[Section 4.3]{BVHP1}) and we know that the associated matrix is antidiagonal
hence non diagonalizable (see \cite[Section 5.1]{BVHP1})

To check diagonalizability we shall use mainly the following

\begin{lem}\label{Dim2IsDiag}
Let $C_j=\{\c_j,\c_{j+(q-1)}\}$ (with $0\leqslant j\leqslant q-2$) be a class with associated matrix $M_j$.
If $M_j\neq 0$, then $M_j$ is diagonalizable if and only if \begin{itemize}
\item[{\bf 1.}] $\binom{k-2-j}{j}\neq 0$ or
\item[{\bf 2.}] $q$ odd and $\det(M_j)\neq 0$.
\end{itemize}
\end{lem}

\begin{proof}
Note that $|C_j|=2$ implies $j+2(q-1)>k-2$. Formula \eqref{Ttcj} yields
\[ U_t\c_j=-(-t)^{j+1} \left[ \binom{k-2-j}{j}\c_j + \binom{k-2-j-(q-1)}{j}\c_{j+(q-1)}\right] \]
and
\begin{align*} U_t\c_{j+(q-1)}= & -t^{j+(q-1)+1}\left[ \binom{k-2-j}{q-1} +(-1)^{j+(q-1)+1} \binom{k-2-j}{j+(q-1)}\right] \c_j \\
\ & -(-t)^{j+(q-1)+1} \binom{k-2-j-(q-1)}{j+(q-1)}\c_{j+(q-1)}\, .
\end{align*}
Now $(-1)^{j+(q-1)+1}=(-1)^{j+1}$ and $k-2<j+2(q-1)$ implies
$\binom{k-2-j-(q-1)}{j+(q-1)}=0$. Therefore the associated matrix is
$M_j=\matrix {(-1)^{j+2}\alpha t^{j+1}}{-\gamma t^{j+q}}{(-1)^{j+2}\beta t^{j+1}}{0}$ with
\[ \alpha:= \binom{k-2-j}{j}, \ \beta:=\binom{k-1-j-q}{j},\   
\gamma:= \binom{k-2-j}{q-1} +(-1)^{j+1} \binom{k-2-j}{j+(q-1)} \]
and characteristic polynomial 
\[ \det(M_j-XI)=X^2-(-1)^{j+2}\alpha t^{j+1} X+(-1)^{j+2}\beta\gamma t^{2j+q+1}\,. \]
If $\alpha\neq 0$, it has distinct roots and $M_j$ is diagonalizable (in any characteristic).\\ 
If $\alpha=0$ and $\det(M_j)\neq 0$, i.e., $\beta\gamma\neq 0$, then the matrix $M_j$
still has distinct eigenvalues in odd characteristic. When $q$ is even we get only one eigenvalue with
multiplicity 2 and the matrix is non diagonalizable 
(note also that the eigenvalue $\sqrt{t^{2j+q+1}}$ is inseparable).
\end{proof}

In this small dimension case, the powers of $t$ have nothing to do with the diagonalizability
of the matrices : hence, from now on, we simply check whether the matrix 
$\left(\begin{smallmatrix}(-1)^{j+2}\alpha & -\gamma \\ (-1)^{j+2}\beta & 0\end{smallmatrix}\right)$ (still denoted $M_j$ by a little abuse of notation)
fits the description of Lemma \ref{Dim2IsDiag}.

To compute the coefficients we shall mainly use the following well known

\begin{lem}\label{KummerThm}(Lucas's Theorem)
Let $n,m\in\N$ with $m\leqslant n$ and write their $p$-adic expansions as
$n=n_0+n_1p+\dots +n_d p^d$, $m=m_0+m_1 p + \dots +m_d p^d$. Then
\[ \binom{n}{m} \equiv \binom{n_0}{m_0} \binom{n_1}{m_1} \dots \binom{n_d}{m_d} \pmod p\,.\]
\end{lem}

\begin{proof}
See, e.g., \cite{DW}.
\end{proof}

\subsection{Odd characteristic}
In this section we assume $p\geqslant 3$ and prove the following

\begin{thm}\label{ThmDiag2OddChar}
Assume $p$ is odd, $k>q+3$ and $|C_j|=2$. Then $M_j$ is diagonalizable.
\end{thm}

\begin{proof}
The hypothesis on the cardinality of $C_j$ yields $k<j+2q$. In particular,
whenever we write $k-2-j=xq+y$ with $0\leqslant y\leqslant q-1$ we have $x=0$ or 1.\\
We split the proof in various cases depending on $j$ and we immediately get rid of the
special case $j=\frac{k-1-q}{2}$ (assuming $k-1-q$ is even) which has already been treated in \cite{BVHP1}.

{\bf The class $C_{\frac{k-1-q}{2}}$.}\\
Since $k=2j+2+(q-1)$ this is a class (the other is the one containing $\c_{\frac{k-2}{2}}$)
associated to the $\Gamma_0(t)$-invariant cusp forms. Moreover $k>q+3$ yields $j> 1$ so the dimension of
the matrix $M_{\frac{k-1-q}{2}}$ is $\leqslant j$. Therefore we can apply \cite[Theorem 5.2]{BVHP1}: the matrix
$M_{\frac{k-1-q}{2}}$ is antidiagonal and diagonalizable (with eigenvalues $\sqrt{t^k}$).

{\bf The class $C_0$.} \\ 
It is easy to see that $\alpha=\beta=1$ and $\gamma=0$, hence $M_0$ is diagonalizable
(with eigenvalues 0 and $t$).

{\bf The class $C_1$.} \\ 
If $\alpha = k-3 \not\equiv 0 \pmod{p}$ we are done ($M_1$ is diagonalizable by Lemma \ref{Dim2IsDiag}).\\
If $k-3\equiv 0\pmod{p}$, then $\beta=k-3-q+1\equiv 1\pmod{p}$ and we need to check $\gamma$. Now
\[ \gamma=\binom{k-3}{q-1}+\binom{k-3}{q}=\binom{k-2}{q}\]
and this is 1 (by Lucas' Theorem) because $q+1<k-2<2q-1$.

{\bf The class $C_j$ with $2\leqslant j < \frac{k-1-q}{2}$.}\\ 
Write $k-2-j=xq+y$ with $0\leqslant y\leqslant q-1$. 
Note that $x=0$ would yield
$k-2-j<q$, i.e., $j>k-q-2$ and $\frac{k-1-q}{2}>k-q-2$, which leads to $k<q+3$ a contradiction.
Therefore $k-2-j=q+y$ and $\alpha\equiv\binom{y}{j}\pmod p$ by Lucas' Theorem. \\
Now $y=q-1$ yields $k=2q+j+1$ a contradiction to $|C_j|=2$, hence $y\leqslant q-2$.
We have $\beta=\binom{y+1}{j}$ and
\begin{align*}
\gamma & =\binom{q+y}{q-1}+(-1)^{j+1}\binom{q+y}{q+j-1}\equiv \binom{y}{q-1}+(-1)^{j+1}\binom{y}{j-1} \pmod{p} \\
\ & \equiv (-1)^{j+1}\binom{y}{j-1} \pmod{p} \,.
\end{align*}
If $\alpha \neq 0$, then $M_j$ is diagonalizable by Lemma \ref{Dim2IsDiag}.\\
If $\alpha =0$, then consider the $p$-adic expansions (with $q=p^r$):
\begin{align*}
y & = y_0+y_1p+\cdots +y_{r-1}p^{r-1};\\
j & = j_\ell p^\ell +j_{\ell+1}p^{\ell+1}+\cdots +j_{r-1}p^{r-1}\quad{\rm (with}\ \ell\geqslant 0\ {\rm and}\ j_\ell\neq 0{\rm )}
\end{align*}
so that $j-1 = (p-1) + \cdots +(p-1)p^{\ell -1}+( j_\ell -1)p^\ell +j_{\ell+1}p^{\ell+1}+\cdots +j_{r-1}p^{r-1}$.
Since $\alpha\equiv \displaystyle{\prod_{i=0}^{r-1} \binom{y_i}{j_i} \pmod{p}}$
and we are assuming $\alpha=0$, there must be at least one index $i$ such that $y_i<j_i$. 
We want to show that (when $\alpha=0$) $\gamma\neq 0$ if and only if $\beta\neq 0$. Indeed 
\[ \gamma\equiv (-1)^{j+1}\prod_{i=0}^{\ell-1} \binom{y_i}{p-1} \cdot \binom{y_\ell}{j_\ell-1}\cdot
\prod_{i=\ell+1}^{r-1} \binom{y_i}{j_i} \not\equiv 0\pmod p \] 
if and only if  $y_i=p-1$ for $0\leqslant i\leqslant \ell-1$, $y_\ell \geqslant j_\ell-1$ and  $y_i\geqslant j_i$ for 
$\ell+1\leqslant i\leqslant r-1$. But this is compatible with $\alpha=0$ if and only if $y_\ell=j_\ell-1$.
An immediate consequence is that 
\[ y=(p-1) + \cdots +(p-1)p^{\ell -1}+( j_\ell -1)p^\ell +y_{\ell+1}p^{\ell+1}+\cdots +y_{r-1}p^{r-1} \]
and 
\[ \beta \equiv \binom{y+1}{j}\pmod{p} \equiv \prod_{i=0}^{\ell-1} \binom{0}{p-1} \cdot \binom{j_\ell}{j_\ell}\cdot
\prod_{i=\ell+1}^{r-1} \binom{y_i}{j_i} \not\equiv 0\pmod p\,.\]
The reverse arrow $\beta \neq 0\Longrightarrow \gamma\neq 0$ is similar. Then, if $\alpha=0$, the block $M_j$ is the 
null matrix or an antidiagonal one and diagonalizable in any case.

{\bf The class $C_j$ with $\frac{k-1-q}{2}< j\leqslant k-q-1$.}\\
Here $j+(q-1)>k-2-j$ yields $\beta=0$ and $\gamma=\binom{k-2-j}{q-1}$. Then $\gamma\neq 0$ if and only if 
$k-2-j=xq+q-1$. Now $x=1$ leads to $k=2q+j+1$, a contradiction to $|C_j|=2$, so $x=0$ and $j=k-1-q$
(the largest index $j$ for which $|C_j|=2$). In this case $\alpha=\binom{q-1}{j}\equiv (-1)^j\pmod{p}$ and
$M_j=\matrix{1}{-1}{0}{0}$
is diagonalizable with eigenvalues 0 and $t^{k-q}$. When $\gamma=0$ the matrix is always diagonalizable
(no matter the value of $\alpha$) and one can actually prove that $M_j=0$. Indeed
$\alpha\equiv\binom{y}{j}\pmod{p}$ and this is 0 for $y<j$. But $j<k-q-1$
yields $k-2-j>q-1$, so $x=1$ and $y=k-2-j-q\geqslant j$ would lead to $j\leqslant \frac{k-2-q}{2}$,
which is outside our current range.  
\end{proof}

This settles all blocks of dimension 2. Note that it means that in odd characteristic $U_t$ is diagonalizable
for any $k\leqslant 2q-1$. For $k=2q$ we only have to check the unique class of cardinality 3, i.e., $C_0$
and we do it here for completeness.

\begin{thm}\label{k=2qOddChar}
Assume $p$ is odd, then $U_t$ is diagonalizable for $k=2q$.
\end{thm}

\begin{proof}
Using formula \eqref{Ttcj} we get $U_t\c_0=t(\c_0+\c_{q-1}+\c_{2q-2})$, $U_t\c_{q-1}=t^q\c_{q-1}$ and 
$U_t\c_{2q-2}=-t^{2q-1}\c_{q-1}$.
The associated matrix has eigenvalues $t^q$, $t$ and 0 with eigenvectors 
$\c_{q-1}$, $\c_0+(1+t^{q-1})\c_{q-1}+\c_{2q-2}$ and $t^{q-1}\c_{q-1}+\c_{2q-2}$.
\end{proof}

\begin{rem}\label{Remq=3}
When $q=3$ we only have the blocks $M_0$ and $M_1$ and both come from $\Gamma_0(t)$-invariant cusp forms.
Some cases of their diagonalizability are treated in \cite{BVHP1} and their characteristic polynomials
for $6\leqslant k\leqslant 62$ can be found in  the file {\verb Blocks_q3.pdf } in the webpage\\ 
{\verb https://sites.google.com/site/mariavalentino84/publications }.
\end{rem}

\begin{ex}
Fix $q=25$: we show the matrices for $M_j^1$ for $k_1=33$ and $M_j^2$ for $k_2=40$. By (the proof of) Theorem 
\ref{ThmDiag2OddChar} we can limit ourselves to $0\leqslant j\leqslant \frac{k_i-1-q}{2}$. 
\[ M_0^1=\matrix {t}{0}{t}{0}\ M_1^1=\matrix {0}{-t^{26}}{-t^2}{0}\ M_2^1=\matrix {t^3}{-t^{27}}{0}{0}\ 
M_3^1=\matrix {-t^4}{2t^{28}}{t^4}{0} \]
\[ M_0^2=\matrix {t}{0}{t}{0} \ M_1^2=\matrix {3t^2}{-t^{26}}{2t^2}{0} \ M_2^2=\matrix {0}{t^{27}}{t^3}{0} 
\ M_3^2=\matrix {0}{0}{0}{0} \]
\[  M_4^2=\matrix {t^5}{- t^{29}}{0}{0} \ M_5^2=\matrix {-t^6}{0}{-t^6}{0} \ M_6^2=\matrix {2 t^7}{t^{31}}{3 t^7}{0}
\ M_7^2=\matrix {0}{- t^{32}}{-t^8}{0} \] 
\end{ex}

\subsection{Even characteristic}\label{SecDiagp=2}
In this section we summarize the results for $q$ even. The matrices for $U_t$ have been described in 
Theorem \ref{ThmDiag2OddChar} and there is nothing new about them, but antidiagonal matrices are not 
diagonalizable in characteristic 2. It is important to notice that they do not appear for even $k$.
\begin{itemize}
\item The block $M_0$ is always diagonalizable. 
\item The block $M_1$ is diagonalizable if and only if $k$ is even (so that $k-3$ is odd and $\alpha\neq 0$).
When $k$ is odd we find 
$M_1=\matrix{0}{t^{q+1}}{t^2}{0}$
which has the inseparable eigenvalue $\sqrt{t^{q+3}}$.
\item The block $M_j$ for $2\leqslant j <\frac{k-1-q}{2}$ is diagonalizable  
unless $\alpha=\binom{k-2-j}{j}=0$ and $\beta=\binom{k-j-q-1}{j}\neq 0$ (which implies $\gamma\neq 0$ as well). 
In that particular case we have
$M_j=\matrix{0}{t^{j+q}}{t^{j+1}}{0}$
with the inseparable eigenvalue $\sqrt{t^{2j+q+1}}$. Note that this cannot happen for even $k$: indeed
$y\equiv j\pmod{2}$ implies that at least one among $\beta\equiv\binom{y+1}{j}\pmod{2}$ and 
$\gamma\equiv\binom{y}{j-1}\pmod{2}$ has to be 0 (hence both are 0 if $\alpha=0$ as well).  
\item The block $M_{\frac{k-1-q}{2}}$ (which is present only for odd $k$) is {\em the} one 
associated to the $\Gamma_0(t)$-invariant cusp forms (see \cite{BVHP1}) and it is antidiagonal of the form
$\matrix{0}{t^{\frac{k-1+q}{2}}}{t^{\frac{k+1-q}{2}}}{0}$
with the inseparable eigenvalue $\sqrt{t^k}$.
\item The block $M_j$ for $\frac{k-1-q}{2}<j<k-q-1$ is always null and 
$M_{k-q-1}=\left(\begin{smallmatrix}t^{k-q}&t^{k-1}\\0&0\end{smallmatrix}\right)$ is diagonalizable.
\end{itemize}
Diagonalizability for the case $k=2q$ follows as in Theorem \ref{k=2qOddChar}.

\begin{ex} For $q=16$ we provide all blocks $M_j^1$ for $k_1=29$ and $M_j^2$ for $k_2=30$ (just for
$0\leqslant j\leqslant \frac{k_i-1-q}{2}$, note that all blocks for $k_2$ are diagonalizable). 
\[ M_0^1=\matrix{t}{0}{t}{0} \ M_1^1=\matrix{0}{t^{17}}{t^2}{0} \ M_2^1=\matrix{0}{t^{18}}{t^3}{0} 
\ M_3^1=\matrix{0}{0}{0}{0} \]  
\[ M_4^1= \matrix{t^5}{t^{20}}{0}{0} \ M_5^1=\matrix{0}{t^{21}}{t^6}{0} \ M_6^1=\matrix{0}{t^{22}}{t^7}{0} \] 
\[ M_0^2=\matrix{t}{0}{t}{0} \ M_1^2= \matrix{t^2}{t^{17}}{0}{0} \ M_2^2= \matrix{t^3}{0}{t^3}{0} \] 
\[ \ M_3^2=M_4^2=\matrix{0}{0}{0}{0} \ M_5^2= \matrix{t^6}{t^{21}}{0}{0} 
\ M_6^2=\matrix {t^7} {0} {t^7} {0} \] 
\end{ex}

\section{Cusp forms for $\G(t)$} \label{SecGamma}
Just like $\G_1(t)$, the group $\G(t)$ has no prime to $p$ torsion and determinant 1, so the dimension of 
$S^1_k(\G(t))$ and $S^2_k(\G(t))$ depend only on the genus of $\G(t)\backslash \Tr$ (which is 0
by \cite[Corollary 5.7]{GN}) and on the number of cusps. The quotient graph $\G(t)\backslash \Tr$ 
has $q+1$ cusps corresponding to $[1:0]$ (the cusp at infinity) and $[r:1]$, $r\in\F_q$. Hence
\begin{itemize}
\item[{\bf 1.}] {$\dim_{\C_\infty} S^1_k(\G(t))=q(k-1)$;}
\item[{\bf 2.}] {$\dim_{\C_\infty} S^2_k(\G(t))= \left\{ \begin{array}{cr}
0 & k=2 \\
q(k-2)-1 & k>2
\end{array}\right. \ .$}
\end{itemize}
Here is a picture of the fundamental domain:

\[ \xymatrix{
\udots  & v_{-1,t}=\matrix {1}{t}{0}{t} \ar[d]^{e_{-1,t}} & v_{-1,0}=\matrix {1}{0}{0}{t} \ar[dl]^{e_{-1,0}} \\
v_{-1,rt}=\matrix {1}{rt}{0}{t} \ar[r]^{e_{-1,rt}} & v_{0,0}=\matrix {1}{0}{0}{1} \ar[r]_{e_{0,0}} &
v_{1,0}=\matrix {t}{0}{0}{1} \\
\ddots& \cdots & v_{-1,(q-1)t}=\matrix {1}{(q-1)t}{0}{t} \ar[ul]^{e_{-1,(q-1)t}} } \]

Put $e_r:=e_{-1,rt}$ for any $r\in \F_q$, so that
$\e_r:=\e_{-1,rt}=\left(\begin{smallmatrix}1&rt\\0&t\end{smallmatrix}\right)
\left(\begin{smallmatrix}0&1\\1/t&0\end{smallmatrix}\right) 
=\left(\begin{smallmatrix}r&1\\1&0\end{smallmatrix}\right)$
(note that $\e_0$ is the fundamental edge $\e$ of Section \ref{SecGamma1}). By harmonicity we have
\begin{equation*}
\c(\e_{0,0})+\sum_{r\in \F_q} \c(e_r)=0,\ {\rm i.e.,}\ 
\c\left( \matrix {0}{t}{1}{0} \right) +
\sum_{r\in \F_q} \c\left( \matrix {1}{rt}{0}{t} \right) =0\,.
\end{equation*}
It is easy to see that the vertex $v_{0,0}$ and all the edges of the fundamental domain are stable.

The basis for our cuspidal forms is
$\{\c_{j,r}\,:\,0\leqslant j\leqslant k-2\,,\, r\in \F_q\}$,
where we put
\[ \c_{j,r}(\e_s)((X-uY)^iY^{k-2-i})= \left\{ \begin{array}{ll} 1 & {\rm if}\ s=r=u\ {\rm and}\ i=j \\
0 & {\rm otherwise} \end{array} \right. \ .\]

To shorten notations let $f^i_u(X,Y):=(X-uY)^iY^{k-2-i}$, then, for the edge at infinity $e_{0,0}$, we have
\begin{equation}\label{UG}
U_t(\c_{j,r}(e_{0,0}))f^i_u(X,Y) = -\sum_{s\in\F_q} U_t (\c_{j,r}(\e_s))f^i_u(X,Y)\\ \nonumber
\end{equation}
and we can recover its value from the others. Hence we only have to compute
$U_t(\c_{j,r}(e_s))f^i_u(X,Y)$ for any $r,s,u\in\F_q$ and $0\leqslant i,j\leqslant k-2$.

\subsection{The case $r\neq 0$}
In this case we transform all $\c_{j,r}(\e_s)$ in $\c_{j,r}(e_0)$
and this produces a lot of zeroes (we use the harmonicity relation
$\c\left(\matrix{0}{1}{t}{0}\right) = 
-\displaystyle{\sum_{v\in \mathbb{F}_q} \c\left(\matrix{1}{0}{vt}{1} \matrix{1}{0}{0}{t}\right) }$,
and then observe that $\bigl(\begin{smallmatrix}1&0\\vt&1\end{smallmatrix}\bigr)\in \Gamma(t)\,$).
\begin{align*}
& U_t\c_{j,r}(\e_s)f^\ell_u(X,Y) = t^{k-m} \sum_{\b\in\F_q}\matrix {1}{-\frac{\b}{t}}{0}{\frac{1}{t}}
\c_{j,r} \left( \matrix {1}{\b}{0}{t} \matrix {s}{1}{1}{0}\right)f^\ell_u(X,Y)\\
 & = t^{k-m} \sum_{\b\in\F_q}\matrix {1}{-\frac{\b}{t}}{0}{\frac{1}{t}}
\c_{j,r}\matrix{s+\b}{1}{t}{0}f^\ell_u(X,Y)\\
 & = t^{k-m} \left\{\matrix {1}{\frac{s}{t}}{0}{\frac{1}{t}} \c_{j,r} \matrix {0}{1}{t}{0} +
\sum_{\b\in\F_q-\{-s\}}\matrix {1}{-\frac{\b}{t}}{0}{\frac{1}{t}}
\c_{j,r} \matrix {s+\b}{1}{t}{0} \right\}f^\ell_u(X,Y) \\
 & = t^{k-m} \bigg\{-\sum_{v\in \F_q} \matrix {1}{\frac{s}{t}}{0}{\frac{1}{t}}
\c_{j,r}\left( \matrix{1}{0}{vt}{1} \matrix {1}{0}{0}{t} \right) \\
\ & + \sum_{\b\in\F_q-\{-s\}}\matrix {1}{-\frac{\b}{t}}{0}{\frac{1}{t}}
\c_{j,r} \left(\matrix{1}{0}{\frac{t}{s+\beta}}{1}\matrix{1}{0}{0}{t}\matrix {s+\b}{1}{0}{-\frac{1}{s+\beta}}\right) \bigg\}
f^\ell_u(X,Y) \\
& = t^{k-m} \bigg\{-\sum_{v\in \F_q} \matrix {1+s v}{\frac{s}{t}}{v}{\frac{1}{t}} \c_{j,r}(e_0) 
+ \sum_{\b\in\F_q-\{-s\}}\matrix {1-\frac{\beta}{s+\beta}}{-\frac{\b}{t}}{\frac{1}{s+\beta}}{\frac{1}{t}}
\c_{j,r}(e_0) \bigg\} f^\ell_u(X,Y) \\
& = - t^{k-1} \sum_{v\in \F_q}\c_{j,r}(e_0)
\left(\left((1+s v)(X-uY)+\frac{s Y}{t}\right)^\ell \left(v(X-uY)+\frac{Y}{t}\right)^{k-2-\ell}\right) \\
\ & + t^{k-1} \!\!\!\!\sum_{\b\in\F_q-\{-s\}}\!\!\!\! \c_{j,r}(e_0)
\left(\left((1-\frac{\beta}{s+\beta})(X-uY)-\frac{\beta Y}{t}\right)^\ell
\left(\frac{X-uY}{s+\beta}+\frac{Y}{t}\right)^{k-2-\ell}\right) \\
& = - t^{k-1} \sum_{v\in \F_q}\c_{j,r}(e_0)
\left[ \sum_{i=0}^\ell \sum_{h=0}^{k-2-\ell} \eta(k,i,h,s,v,t) f^{i+h}_u(X,Y)\right] \\
& + t^{k-1}\sum_{\b\in\F_q-\{-s\}}\c_{j,r}(e_0)
\left[ \sum_{i=0}^\ell \sum_{h=0}^{k-2-\ell} \mu(k,i,h,s,\beta,t) f^{i+h}_u(X,Y)\right]
\end{align*}
Now, since $\c_{j,r}(e_0)f^{i+h}_u(X,Y)=0$
(no matter the values of $i$, $h$ and $u$, since $r\neq 0$), this is always 0.

\subsection{The case $r=0$} Note that the $\c_{j,0}$ are basically the $\c_j$ of Section \ref{UBasis}, indeed
\[ \c_{j,0}(\e_0)f^i_0(X,Y)=\c_j(\e)f^i_0(X,Y)\quad {\rm for\ any\ } 0\leqslant i,j,\leqslant k-2 \]
and $\c_{j,0}$ is trivial everywhere else.

Therefore we have $U_t\c_{j,0}(\e_0)f^i_0(X,Y)=U_t\c_j(\e)f^i_0(X,Y)$
and we can use formula \eqref{Ttcj} for it. While, with computations as above, for any $s,u\in \F_q^*$ we
find that $U_t\c_{j,0}(\e_s)f^i_u(X,Y)$ is equal to $\c_{j,0}(\e_0)$ evaluated at a linear combination
of powers of $(X-uY)$ times powers of $Y$ (hence 0 since $u\neq 0$).

\begin{thm} \label{ThmDiagGammaGamma1}
The operator $U_t$ acting on $S^1_k(\Gamma)$ is diagonalizable if and only if the operator $U_t$ acting on
$S^1_k(\Gamma_1)$ is diagonalizable.
\end{thm}

\begin{proof}
Order the given basis in the following way
\[ \mathcal{B}_k^1(\G):=\{ \c_{0,0}, \c_{1,0}, \dots ,\c_{k-2,0},\c_{j,r} : 0\leqslant j\leqslant k-2\,,\,
r\in\F_q^*\,\} \]
(actually we do not mind about the order of the last $(k-1)(q-1)$ terms we only need the first $k-1$ to be
$\c_{0,0}, \c_{1,0}, \dots ,\c_{k-2,0}$).
For any $r\neq 0$ and any $0\leqslant j\leqslant k-2$, we have seen that $\c_{j,r}\in Ker(U_t)$
while $U_t\c_{j,0}$ is a linear combination of the $\c_{i,0}$ and the coefficients are exactly the same
appearing in the formula for $U_t\c_j$ in terms of the $\c_i$ in Section \ref{SecGamma1}.
Hence the matrix associated to $U_t$ with respect to the basis $\mathcal{B}_k^1(\G)$ is
\[ M_{\mathcal{B}_k^1(\G)}:=\matrix{A}{0}{0}{0} \]
where $A$ is the square matrix of dimension $k-1$ associated to the action of $U_t$ on $S^1_k(\Gamma_1)$
and the others are null matrices (of the right dimension). Obviously this matrix is diagonalizable if and only if $A$
is diagonalizable.
\end{proof}

\begin{rem}
The basis for $S^2_k(\Gamma)$ can be described (and ordered) in the following way
$\mathcal{B}_k^2(\G):=\{ \c_{1,0}, \dots ,\c_{k-3,0},\c_{j,r} : 0\leqslant j\leqslant k-3\,,\,
r\in\F_q^*\,\}$.
Hence the matrix associated to $U_t$ with respect to $\mathcal{B}_k^2(\G)$ is
$M_{\mathcal{B}_k^2(\G)}:=\matrix{A'}{0}{0}{0}$
where $A'$ is associated to the action of $U_t$ on $S^2_k(\Gamma_1)$ and its diagonalizability depends on
the analogous property for $A'$. 
\end{rem}

\end{document}